\RequirePackage{fix-cm}
\documentclass[smallextended]{svjour3}       
\smartqed  
\usepackage{graphicx}
 \usepackage{mathptmx}      
\usepackage{amsmath, amssymb, amsfonts, enumerate, color}
\usepackage{graphicx}
\usepackage{epstopdf}
\usepackage{caption}
\usepackage{subcaption}
\newcommand{\R}{\mathbb{R}}
\newtheorem{algorithm}{Algorithm}[section]
\begin{document}

\title{An extragradient algorithm\\ for quasiconvex equilibrium problems  without monotonicity\thanks{The research of the first author was supported by the Vietnam Academy of Science and Technology under Grant Number CTTH00.01/22-23.}
}

\titlerunning{Mathematical programming techniques}        

\author{Le Hai Yen \and Le Dung Muu 
}


\institute{Corresponding author: L.H.Yen \at
              Affiliation: Institute of Mathematics, VAST, Hanoi, Viet Nam\\
              \email{lhyen@math.ac.vn}           
           \and
           L.D.Muu \at
             Affiliation: TIMAS, Thang Long University and Institute of Mathematics, VAST, Hanoi, VietNam
              \\\email{ldmuu@math.ac.vn}
}

\date{Received: date / Accepted: date}

\maketitle

\begin{abstract}
We attempt to provide  an   algorithm for approximating a solution of   the  quasiconvex    equilibrium problem  that was proved to exist a solution by  K. Fan 1972. The proposed algorithm is an iterative procedure, where the search direction at each iteration is a normal-subgradient, while the step-size is updated avoiding Lipschitz-type conditions. The algorithm is convergent to a $\rho$- quasi-solution with any positive $\rho$ if  the bifunction $f$ is semistrictly quasiconvex in its second variable, while it converges to the solution when $f$  is strongly quasiconvex. Neither monotoniciy nor Lipschitz property is required.  
\keywords{ Equilibria \and Quasiconvexity \and Normal subgradient \and Linesearch}
\subclass{90C33 \and 65K10 \and 90C26}
\end{abstract}

\section{Introduction}
\label{intro}
Let $C$ be a nonempty closed convex set in $\R^n$ and $f: \R^n \times \R^n \to \R$ be a given bifunction such that $f(x,x) = 0$ for every $x \in C$. We consider the problem
$$ \textnormal{Find } \ x^*\in C:\  f(x^*,y)  \geq 0 \quad \forall y\in C.  \eqno(EP)$$
In what follows we call Problem (EP) a convex (resp. quasiconvex) equilibrium problem if the function $f(x,.)$ is convex (resp. quasiconvex) on $C$ for any $x\in C$.
The inequality appeared in Problem (EP) first was used by Nikaido and Isoda in 1955 \cite{Ni1} in a non-cooperative convex game.  
In recent years this problem attracted much attention of many authors as it contains a lot numbers of important problems such as optimization, variational inequality, Kakutani fixed point, Nash equilibrium problems and others as special cases, see e.g. the interesting monographs \cite{Bi2018,Ko2001},   the papers \cite{Bi2015,HD2020,HMS2020,HM2011,Ma2000,MQ2009,QMH2008,So2011,St2016} and the references cited therein.

Many algorithms have been developed for solving (EP) under the assumption that the bifunction is convex and subdifferentiable with respect  to the second variable while the first one being fixed.  Almost all of these algorithms are based upon the auxiliary problem principle, which states that  when $f(x,.)$ is convex, subdifferentiable  on $C$, then the solution-set of (EP) coincides with that of the regularized  problem
$$\text{find} \ x^* \in C: f_\rho(x^*,y):= f(x^*,y) + \frac{1}{2\rho}\|y-x^*\|^2 \geq 0 \  \forall y\in C, \eqno(REP) $$
with any $\rho > 0$.  The main advantage of the latter problem is that  the regularized bifunction $f_\rho$ is strongly convex in the second variable when the first one being fixed.

A basic method for solving Problem (REP) is the extragradient one, where at each iteration $k$, having $x^k\in C$, a main operation is of solving the mathematical subprogram 
$$\min\{ f_\rho (x^k,y):= f(x^k,y) + \frac{1}{2\rho}\|y-x^k\|^2: y \in C\}. \eqno(MP)$$
Thanks to convexity of the function $f(x^k,.)$ this problem  is a strongly convex program, and therefore it is uniquely solvable.
However, when $f(x,.)$ is quasiconvex rather convex, Problem (MP), in general, is  not strongly convex, even not quasiconvex.

In the seminal paper \cite{Fa1972} in 1972,  K. Fan called  Problem (EP) a minimax inequality and established solution existence results for it, when $C$ is convex, compact and $f$ is quasiconvex on $C$.  To our best knowledge, up to now there does not exist an algorithm for finding a solution of the problem considered in \cite{Fa1972} by K. Fan.

  It worth mentioning that
  when  
$f(x, ·)$ is convex and subdiﬀerentiable on $C$, the equilibrium problem (EP) can be reformulated as the following multivalued variational inequality 
$$ \text{find} \ x^* \in C, v^* \in F(x^*) : \langle v^*, x - x^*\rangle \geq 0 \quad \forall x\in C, \eqno(MVI)$$
 where $F(x^*) = \partial_2 f(x^*,x^*)$  with $\partial_2 f (x^*, x^*)$ being the diagonal subdiﬀerential of $f$ at $x^*$, that is the subdifferential of the convex function $f (x^*, ·)$ at $x^*$. In the case $f(x, ·)$ is semi-strictly quasiconvex rather
than convex, Problem (EP) can take the form of (MVI) with $F(x) := Na_{f(x,x)} \setminus \{0\}$,
where $ Na_{f(x,x)}$ is the normal cone of the adjusted sublevel set of the function $f(x, ·)$
at the level $f(x, x)$, see \cite{Au1}. More details about the links between equilibrium
problems and variational inequalities can be found in \cite{Au2}. 

Based upon the auxiliary principle, different methods   such as the fixed point, projection, extragradient, regularization, gap function  ones have been developed
for solving equilibrium problem (EP) by using mathematical programming techniques, where the bifunction involved possesses certain monotonicity properties. Almost all of them require that the bifunction is convex with respect to its second variable, see e.g. the comprehensive monograph \cite{Bi2018} and the references therein. 

In \cite{CLS2016}, the authors studied an infeasible interior proximal algorithm for
solving quasiconvex equilibrium problems with polyhedral constraints. At each iteration $k$  of this algorithm, having $x^k$  it requires  globally solving  a nonconvex mathematical programming problem, where the objective function is the sum of $f(x^k,.)$ and a strongly convex function defined by a distance function. The convergence of this algorithm is proved under an assumption depending on the iterates $x^k$ and $x^{k+1}$. Very recently, Iusem and Lara \cite{IL2021} propose  an  algorithm for solving quasiconvex equilibrium problem (EP). Their algorithm can be considered as a standard proximal point method for optimization problem applied to the quasiconvex function $f(x,.)$. The convergence    has been proved when $f$ is pseudomonotone, Lipschitz-type and strongly quasi-convex.

  In our recent papers \cite{YM2020,YM2021}, by using the normal subdifferential of  quasiconvex functions, we have proposed projection algorithms for Problem (EP) when the bifunction is pseudo and paramonotone.  

In this paper, we continue our work by modifying  the linesearch extragradient algorithm commonly used for   convex equilibrium problem (EP) to solve quasiconvex equilibrium problems without requiring any monotonicity and Lipschitz-type  properties of the bifunction involved. More precisely, after the next section that contains preliminaries on normal subdifferentials of a quasiconvex function, in the third section, we describe an extragradient linesearch algorithm for this quasiconvex  equilibrium problem. Then by observing that the solution set  of the regularized problem coincides with  that of the Minty (dual)  one for semi-strictly quasiconvex bifunction, we  prove that the algorithm   converges to a quasi (prox) solution when the bifunction involved is semi-strictly quasiconvex in its second variable, which is the unique solution when the bifunction is strongly quasiconvex in its second variable.    We close the paper by presenting some computational results showing the efficiency and behavior of the proposed algorithm.

\section{Preliminaries on quasiconvexity, normal subdifferentials and monotonicity}
\begin{definition}(\cite{Au2,GL2021,Ma1969})
Let $C$ be a convex set in $\R^n$.
Let $\varphi:\R^n \rightarrow \R \cup \{+\infty\}$ such that $C \subseteq dom \varphi$. The function $\varphi$ is said to be
\begin{itemize}
\item[(i)] \textbf{quasiconvex} on $C$  if and only if for every $x,y \in C$ and $\lambda \in \left[0,1\right]$, one has
\begin{equation}
\varphi[(1-\lambda)x +\lambda y] \leq \max[\varphi(x), \varphi(y)].\label{eq1}
\end{equation}
\item[(ii)] \textbf{semi-strictly quasi-convex} on $C$ if it is quasiconvex and for any  every $x,y \in C$ and $\lambda \in \left(0,1\right)$, one has
\begin{equation}
	\varphi(x) <\varphi(y) \Rightarrow \varphi[(1-\lambda)x +\lambda y]< \varphi(y).
\end{equation}

\item[(iii)] \textbf{strongly quasiconvex} on $C$ with modulus $0< \gamma < \infty$ if  for every $0\leq \lambda \leq 1$  
$$ \varphi(\lambda x +(1-\lambda) y) \leq \max\{\varphi(x), \varphi(y)\} -\lambda (1-\lambda) \frac{\gamma}{2}\|x-y\|^2 \ \forall x, y \in C.$$
\item[(iv)] \textbf{essentially quasiconvex} on $C$ if it is quasiconvex and every its local minimum is a global one.
\item[(v)] \textbf{pseudoconvex} on on $C$ if it differentiable on an open set containing $C$ and
$$\langle \nabla \varphi (x), y-x\rangle \geq 0 \Rightarrow \varphi (y) \geq \varphi (x) \ \forall x, y \in C.$$ 
\item[(vi)] \textbf{ proximal convex} on $C$ with modulus  $\alpha > 0$ (shortly $\alpha$-prox-convex) if  
$prox_{\varphi}(C,z) \not=\emptyset$ 
and there exists $\alpha > 0$ such that
$$ \ p \in prox_{\varphi}(C,z) \Rightarrow \alpha \langle x-p, z-p\rangle \leq \varphi (x)- \varphi(p) \  \forall x\in C,$$
where $prox_{\varphi}(C,z)$ is the proximal mapping of $\varphi$ at $z$ on $C$, that is
$$prox_{\varphi}(C,z) := \text{argmin}\{\varphi(x) := \varphi (y)+ \frac{1}{2}\|y-x\|^2: y\in C\}.$$
\end{itemize}
\end{definition}
It is well known that strongly quasiconvex $\Rightarrow$  semi-strictly quasiconvex $\Rightarrow$ essentially quasiconvex  $\Rightarrow $ quasiconvex.

Clearly,  $\varphi$ is quasiconvex if and only if, for every $\alpha \in \R$, the strict level set at the level $\alpha$, that is 
$$L_\alpha: =\{ y: \varphi (y) < \alpha\} $$
is convex .

 Recall that the (Hadamard)  directional derivative of a   function $\varphi$ at $x$ with direction $d  $ is defined as
 $$\varphi'(x,d) := \lim\inf_{t \searrow  0, u\to d} \frac{\varphi(x + tu) -f(x)}{t}.$$
 A point $x\in C$ is said to be a stationary point of $\varphi$ on $C$ if $\varphi'(x,d) \geq 0$ for every $d$. A point  is  a minimizer of $\varphi$ on $C$ then it is a stationary point. The converse direction is true when $\varphi$ is convex or pseudo convex on $C$.

 The Greenberg-Pierskalla   subgradient  of a quasiconvex function \cite{GP1}  is defined as
$$\partial^{GP} \varphi(x):= \{  g \in \R^{n} : \langle g, y - x\rangle > 0  \Rightarrow \varphi (y) \geq \varphi(x)\}.$$
A variation of this subdifferential is the star-subdifferential that is defined as
$$\partial^*\varphi (x):= \{ g\in \R^{n}: \langle g,y-x\rangle {\color{blue}<} 0 \ \forall y\in L_\varphi(x)\},$$
where $L_\varphi(x)$ stands for the strict level set of $\varphi$ with level $\varphi (x)$.  
 It is well known \cite{GP1,Pe1}  that if $\varphi$ is continuous on $\R^n$, then $\partial^*\varphi(x)$ contains nonzero vector and 
$$\partial^*\varphi (x) \cup\{0\} = cl(\partial^* \varphi(x) ) = \partial^{GP}\varphi (x),$$ 
where $cl(A)$ stands for the closure of  the set $A$. Thus the star-subdifferential is also called the normal-subdifferential. Various calculus rules for normal subdifferential can be found in \cite{Pe1}.
  
  The following concepts are commonly used in the field of equilibrium problem \cite{Bi2018}.
  
\begin{definition} Let $f: C\times C \to \R$ and $S\subset C$
\begin{itemize}
\item[(i)] $f$ is said to be strongly monotone on $S$ with modulus $\eta \geq 0$ (shortly $\eta$-strongly monotone) if
 $$f(x,y) + f(y,x) \leq - \eta \|x-y\|^2 \  \forall x, y \in S.$$
  If $\eta = 0$ it is also called monotone on $S$.
	   
\item[(ii)] $f$ is said to be paramonotone on $C$ if $x$ is a solution of (EP) and $y\in C$, $f(x,y) = f(y,x) = 0$ then $y$ is also a solution of (EP).

\item[(iii)]  $f$ is said to be  pseudomonotone on $C$ if $f(x,y)  \geq 0$ then $f(y,x)\leq 0$ for every $x,y \in C$. 

\item[(iv)]  $f$ is said to be Lipschitz-type on $C$ if
$$f(x,y) +f(y,z) \geq f(x,z) -L_1\|x-y\|^2 - L_2\|y-z\|^2 \ \forall x,y,z \in C.$$
Clearly, in the case of optimization when $f(x,y) := \varphi (y) - \varphi (x)$ it possesses  both the paramonotonicity and Lipschitz-type property.

\end{itemize} 

\end{definition}

\section{Algorithm and its convergence} A problem closely related to Problem (EP) is the  Minty (or dual) equilibrium  one that  is defined
as
$$\text{Find } z^*\in C \text{ such that } f(y, z^*) \leq 0. \eqno{(DEP)}$$
Let us denote by $S$ and $S_d$ the solution set of (EP) and (DEP) respectively. It is clear that if $f$ is pseudomonotone on $C$ then $S\subseteq S_d$. Conversely, $S_d \subseteq S$ if $f$ is upper semi-continuous with respect to the first variable and convex with respect to the second variable (see \cite{Mu1984}). 

In what follows we always suppose that
 $f(.,y)$ is upper semi-continuous for any $y\in C$.

In the following lemma, we   prove that the inclusion $S_d \subseteq S$  still holds true when $f$ is semi-strictly quasiconvex with respect to the second variable.

\begin{lemma}\label{M1}
	Assume that
	 $f(x,.)$ is   semi-strictly quasiconvex  on $C$ for any $x\in C$.
		 Then $S_d \subseteq S$.\label{lem1}
\end{lemma}
\begin{proof}
	Let $z^*\in S_d$. If $z^*\not\in S$, then there would exist  $y\in C$ such that $f(z^*,y)<0$.
	
	For $\lambda \in (0,1)$, set $y_{\lambda}= \lambda z^*+(1-\lambda)y$.  Since $f(.,y)$ is upper semi-continuous, there exists $0<\lambda<1$ such that $f(y_{\lambda},y)<0$.
		Since  $z^*\in S_d$, $f(y_{\lambda}, z^*)\leq 0$.
		
	We consider two cases
	\begin{itemize}
		\item[$\bullet$] Case 1: $f(y_{\lambda}, z^*)< 0$. By the quasiconvexity of $f(y_{\lambda},.)$,
		$$0=f(y_{\lambda}, y_{\lambda})\leq \max\{f(y_{\lambda}, z^*), f(y_{\lambda}, y)\}<0.$$
		This is a contradiction.
			\item[$\bullet$] Case 2: $f(y_{\lambda}, z^*)= 0$. By the semi-strictly quasiconvexity of $f(y_{\lambda},.)$ and the fact that $f(y_{\lambda},y)<0=f(y_{\lambda}, z^*) $, which would imply
		$$0=f(y_{\lambda}, y_{\lambda})< \max\{f(y_{\lambda}, z^*), f(y_{\lambda}, y)\}=0.$$
		This is also a contradiction.
	\end{itemize}

\end{proof}
 
The following algorithm can be considered as a modification of the one in \cite{QMH2008}   for solving Problem (EP) when the bifunction is quasiconvex with respect to its second variable.\\
 
\begin{algorithm}
	Take $\alpha,\theta\in (0,1)$  and two   sequences  $\{\rho_k\}_{k\geq 0}$,$\{\sigma_k\}_{k\geq 0}$ of positive numbers such that
	\begin{eqnarray}
		 \rho_k \  \text{ nonincreasingly converges to some} \ \overline{\rho}  > 0, \nonumber\\
		 \quad\sum_{k=0}^{\infty} \sigma_k=\infty, \quad \sum_{k=0}^{\infty} \sigma_k^2 <\infty.\nonumber
	\end{eqnarray}
	\textbf{Initialization:} Pick $x^0\in C$.
	\\\textbf{Iteration $k = 0,1...$}
	\begin{itemize}
		\item[$\bullet$] Find $y^k$ such that
		\begin{equation}
			y^k \in \text{argmin}_{y\in C} \left\{ f(x^k,y) +\frac{1}{2\rho_k} \|x^k-y\|^2\right\}.
			\label{alg1}
		\end{equation}
		\item[$\bullet$] If $y^k=x^k$, then \textbf{stop}: $x^k$ is a stationary point or a solution.\\
 If $y^k\not=x^k$, find the smallest positive integer $m$ such that $z^{k,m}=(1-\theta^m)x^k +\theta^m y^k$ and
		\begin{equation}
			f(z^{k,m},x^k) -f(z^{k,m},y^k) \geq \frac{\alpha}{2\rho_k} \|y^k-x^k\|^2, \label{alg2}
		\end{equation}
	and set $z^k=z^{k,m}$.
		\item[$\bullet$] Take 
		\begin{equation}
			g^k \in \partial^*_2f(z^k,x^k):=\left\{g\in \R^n: \langle g, y-x^k \rangle<0 \text{ if } f(z^k,y)< f(z^k,x^k)\right\},\label{alg3}
		\end{equation}
	and normalize it to obtain $\|g^k\|=1$ ($g^k \not = 0$, see Proposition 3.2 below).
	\\ Compute 
	\begin{equation}
		x^{k+1}=P_C(x^k -\sigma_k g^k). \label{alg4}
	\end{equation}

	\end{itemize}
If $x^{k+1}=x^k$ then \textbf{stop}: $x^k$ is a solution, else set $k:=k+1$.

\end{algorithm}

\begin{remark}\label{3.1}
\begin{itemize}
\item[(i)] The existence of solution for (\ref{alg1}) can be guaranteed under the assumption that the function $f(x^k,.)$ is lower semicontinuous and $2-$weakly coercive (see \cite{GL2021}), \textit{i.e.},
$$\liminf_{\|y \|\rightarrow +\infty} \frac{f(x^k,y)}{\|y\|^2} \geq 0.$$
If $C$ is bounded, the $2-$weakly coercivity assumption can be dropped. If $f(x,.)$ is strongly quasi-convex on $C$, then it is 
$2-$weakly coercive (see \cite{IL2021} Lemma 2).

 Another example is the $\alpha$-proximal convex function introduced in \cite{GL2021}, where it  has been proved  that if $h$ is $\alpha$-proximal  convex on $C$ then $Prox_h(C,z)$ is a singleton. It can be seen from Lemma 2.2 in \cite{IL2021} that if $f(x,.)$ is strongly quasi-convex on $C$ for any $x\in C$, then it is proximal convex on $C$.
 Note that if $h$ is convex, then   $\alpha =1$.  

\item[(ii)] When $f$ is quasiconvex rather convex on $C$,   problem (\ref{alg1}), in general is not convex, even not quasiconvex. However, in some special cases (see examples below) one can choose reqularization parameter $\rho_k$ such that problem (\ref{alg1}) is strongly convex, and therefore it is uniquely solvable.

\end{itemize}

\end{remark}

In contrast to the convex case, in the algorithm, $y^k = x^k$ does not necessarily implies that $x^k$ is a solution. But,  part (i) of the following proposition shows that it is a solution 
restricted on a part of $C$, while in the rest part it is only a stationary point.
                                                   
\begin{proposition}
	Suppose that  $y^k=x^k$. 
	\begin{itemize}
	\item[(i)] If $x^k$ is not a solution of (EP), that means
	\begin{equation}
		\Omega(x^k):=\left\{ y\in C: f(x^k,y) <0\right\}.\nonumber
	\end{equation} is nonemty, then for any $y\in \Omega(x^k)$
$$f'_{x^k}(x^k,y-x^k)=0,$$
where $f_{x^k}:= f(x^k,.)$.
 \item[(ii)] If $f(x^k,.)$ is pseudoconvex on $C$ or strongly quasiconvex  on $C$, then  $x^k$ is a solution of (EP).
 \end{itemize}\label{prop2}
\end{proposition}
 
\begin{proof}
\begin{itemize}
\item[(i)] For $y\in \Omega(x^k)$, set $d=y-x^k$. For $\lambda\in \left(0,1\right)$, set $y_{\lambda}= x^k +\lambda d = \lambda y +(1-\lambda)x^k$. Since $f(x^k,x^k) = 0$, by the semi-strictly quasiconvexity of $f(x^k,.)$, we have
$$f(x^k, y_{\lambda})< 0.$$
So,
\begin{equation}
	\frac{f(x^k,x^k+\lambda d) -f(x^k,x^k)}{\lambda}<0.\label{prop11}
\end{equation}

From $$x^k = y^k\in \text{argmin}\{ f(x^k,y) + \frac{1}{2\rho_k}\|y-x^k\|^2: y\in C\},$$ it follows that for any $y\in C$,
\begin{equation}
	f(x^k,y) +\frac{1}{2\rho_k} \| y -x^k\|^2\geq 0.\nonumber
\end{equation}
Let $y=y_{\lambda}$, then
\begin{equation}
	f(x^k,y_{\lambda}) +\frac{1}{2\rho_k} \lambda^2\nonumber \|y- x^k\|^2\geq 0.
\end{equation}
Therefore,
\begin{equation}
	\frac{f(x^k,x^k+\lambda d) -f(x^k,x^k)}{\lambda}\geq -\lambda\frac{1}{2\rho_k}\|d\|^2.\label{prop12}
\end{equation}
By combining (\ref{prop11}) and (\ref{prop12}) and let $\lambda \rightarrow 0^+$, we obtain
$f'_{x^k}(x^k,d)=0.$

\item[(ii)] Now, assume that $f(x^k,.)$ is pseudoconvex on $C$, then $f(x^k,.)$ is differentiable on an open set containing $C$ and for any $y,y'\in C$, we have 
$$\nabla_2 f(x^k,y)(y'-y) \geq 0 \Rightarrow f(x^k,y') \geq f(x^k,y).$$
From $x^k=y^k$, it implies that $\langle\nabla_2f(x^k,x^k),y-x^k \rangle \leq 0$ for every $y\in C$. Therefore, $f(x^k,y) \geq 0$ for $y\in C$.

 If  $f(x^k,.)$ is strongly quasiconvex,  then subproblem (\ref{alg1})    is uniquely solvable. Since $x^k =y^k$ with $y^k$ being the solution of subproblem (\ref{alg1}), we have
\begin{equation}\label{M3}
 0 \leq f(x^k, y) + \frac{1}{2\rho_k} \|y-x^k\|^2  \ \forall y\in C.
\end{equation}
Let $ y:= \lambda x + (1- \lambda )x^k$ with any $x\in C$ and $ \lambda \in [ 0, 1]$. Then applying (\ref{M3}),
 by the strong quasiconvexity of $f(x^k,.)$,  we obtain
$$0\leq f(x^k, \lambda x + (1- \lambda )x^k)+ \frac{1}{2\rho_k}\|\lambda x + (1- \lambda )x^k - x^k\|^2 $$
$$\leq \max\{f(x^k,x^k), f(x^k,x) \} -\lambda (1-\lambda)\frac{\gamma}{2}\|x-x^k\|^2  +\frac{1}{2\rho_k}\|\lambda x + (1- \lambda )x^k-x^k\|^2.$$
Thus for any $\lambda \in [0,1] $ we have
$$0\leq \max\{f(x^k,x),0\}+\Big [ \frac{\lambda^2}{2\rho_k} - \gamma \frac{\lambda(1-\lambda)}{2} \Big] \|x-x^k\|^2\  \forall x\in C.$$
Since  $\rho_k \searrow  \overline{\rho} > 0$, one can choose $\lambda > 0$ small enough  so that  
$ \frac{\lambda^2}{2\rho_k} - \gamma \frac{\lambda(1-\lambda)}{2} < 0 $. Hence
$f(x^k,x) \geq 0$ for every $x\in C$.

\end{itemize}
\end{proof}
\begin{proposition}
Assume that $f(.,y)$ is continuous on $C$ for any $y\in C$. If $y^k \not=x^k$ then the following statements hold:
	\begin{itemize}
		\item[(i)] There exists a positive integer $m$ satisfying (\ref{alg2}).
\item[(ii)] If   $f(x,.)$ is semi-strictly quasiconvex on $C$ for any $x\in C$, then  $f(z^k,x^k) >0$. 
		\item[(iii)] $0\not\in \partial^*_2 f(z^k,x^k)$. 
	\end{itemize}
\label{prop3}
\end{proposition}

\begin{proof}
	\begin{itemize}
		\item[(i)] If there does not exist $m$ satisfying (\ref{alg2}),  then for every positive integer $m$, we have
		\begin{equation}
			f(z^{k,m},x^k) -f(z^{k,m},y^k) < \frac{\alpha}{2\rho_k}\|y^k-x^k\|^2. \label{eq8}
		\end{equation}
		Let $m \rightarrow +\infty$,  we have $z^{k,m} \rightarrow x^k$ and (\ref{eq8}) becomes
		\begin{equation}
		 	-f(x^k,y^k) \leq \frac{\alpha}{2\rho_k} \|y^k-x^k\|^2. \label{eq9}
		\end{equation}
		On the other hand, (\ref{alg1}) means that for all $y\in C$,
		\begin{equation}
			f(x^k,y^k) +\frac{1}{2\rho_k} \|y^k-x^k\|^2 \leq f(x^k,y) +\frac{1}{2\rho_k} \|y-x^k\|^2.\nonumber
		\end{equation}
		By choosing $y=x^k$, we obtain
		\begin{equation}
			f(x^k,y^k) +\frac{1}{2\rho_k} \|y^k-x^k\|^2 \leq 0.\label{eq10}
		\end{equation}
		Combining (\ref{eq9}) with (\ref{eq10}), it follows that $\alpha\geq 1$. This is a contradiction because $\alpha \in \left(0,1\right)$.

		\item[(ii)] From (\ref{alg2}), $f(z^k,x^k)>f(z^k,y^k)$. By the semi-strictly quasiconvexity of $f(z^k,.)$ on $C$, it follows
		\begin{equation*}
			0=f(z^k,z^k)< f(z^k,x^k).
		\end{equation*}
	
		\item[(iii)] It follows from part (ii) that 
			\begin{equation*}
			0=f(z^k,z^k)< f(z^k,x^k).
			\end{equation*}
		By the definition of $\partial^*_2f(z^k,x^k)$, it is clear that $0\not\in \partial^*_2f(z^k,x^k)$.

	\end{itemize}
\end{proof}

\begin{proposition}
	If $x^{k+1}=x^k$ then $z^k$ is a solution of (EP) provided $f(x,.)$ is semi-strictly quasiconvex on $C$ for any $x\in C$.
\end{proposition}
\begin{proof}
	By the algorithm, $x^{k+1}=x^k$  means that $x^k=P_C(x^k-\sigma_k g^k),$ which is equivalent to
	\begin{equation}
		\langle g^k, y-x^k\rangle \geq 0 \quad\forall y\in C.\label{eq11}
	\end{equation}
Remember that, by (\ref{alg3}),
\begin{equation*}
	g^k \in \partial^*_2f(z^k,x^k):=\left\{g\in \R^n: \langle g, y-x^k \rangle<0 \text{ if } f(z^k,y)< f(z^k,x^k)\right\}.
\end{equation*}
Thus, by (\ref{eq11}) ,   $f(z^k,y) \geq f(z^k,x^k)$ for $y\in C$. 

Note that, in part (ii), Proposition \ref{prop3}, we have  proved that if $x^k\not=y^k$,  then $f(z^k,x^k )>0$. So, we can conclude that  $f(z^k,y) \geq f(z^k,x^k) \geq 0$ for every $y\in C$, which means that  $z^k$ is a solution of (EP).

\end{proof}

\begin{proposition} Suppose that the solution-set $S_d$ of the Minty problem is nonempty.
	Let $z^* \in S_d$, then 
\begin{equation}
	\|x^{k+1}-z^*\|^2 \leq \|x^k-z^*\|^2 +\sigma_k^2,
	\label{eq12}
\end{equation}
and
\begin{equation}
	\liminf_{k \rightarrow +\infty} \langle g^k, x^k -z^*\rangle =0.\label{eq13}
\end{equation}\label{prop5}
\end{proposition}
\begin{proof}
	For $y\in C$, we have
	\begin{eqnarray}
		\|x^{k+1} -y\|^2&=& \|P_C(x^k -\sigma_k g^k)-y\|^2\nonumber\\
		&\leq& \|x^k -\sigma_k g^k -y\|^2\nonumber\\
		&\leq& \|x^k -y\|^2 +\sigma_k^2 + 2\sigma_k \langle g^k, y- x^k \rangle\nonumber.
	\end{eqnarray}
 With $y=z^*\in S_d$, we  have
\begin{equation}
	\|x^{k+1} -z^*\|^2 \leq \|x^k -z^*\|^2 +\sigma_k^2 + 2\sigma_k \langle g^k, z^*- x^k \rangle\label{eq14}.
\end{equation}
Since $f(z^k,z^*)\leq 0< f(z^k,x^k)$ and $g^k \in \partial^*_2f(z^k,x^k)$, it follows that
\begin{equation*}
	\langle g^k, z^*-x^k \rangle <0.
\end{equation*}
Therefore, 
\begin{equation*}
	\|x^{k+1}-z^*\|^2 < \|x^k-z^*\|^2 +\sigma_k^2.
\end{equation*}
\end{proof}
 
\vskip1cm
Following \cite{Go2018} we say that a point $x \in C$ is $\rho$- quasi-solution (prox-solution) to Problem (EP)  if $f(x,y)+\frac{1}{2\rho}\|y- x\|^2 \geq 0$ for every $y\in C$.

For the convergence of the proposed algorithm we need the following assymptions.
\begin{itemize}
\item[(A0)] $f$ is continuous jointly in both variables on an open set containing $C\times C$;
\item[(A1)] $f(x,.)$ is semistrictly quasiconvex   on $C$ forevery $x\in    C$;
\item[(A2)] the solution-set $S_d$ of the Minty problem is nonempty;
	\item[(A3)] The sequence $\{y^k\}$ is bounded.
\end{itemize}

\begin{theorem}
	Suppose that the algorithm does not terminate. Let $\{x^k\}$ be the infinite  sequence generated the algorithm. Under the assumptions  (A0),(A1),(A2),(A3), there exists a subsequence of $\{x^k\}$  converging to a $\overline{\rho}$- quasi solution $\overline{x}$.
	 If   in addition, $f(x,.)$ is strongly quasi-convex for every $x\in C$,  then $\{x^k\}$ converges to the unique solution of (EP).

\end{theorem}

\begin{proof}
	Let $z^* \in S_d$. By part (i) Proposition \ref{prop5} and $\sum_{k=1}^{+\infty} \sigma^2_k<+\infty$, the sequence $\{\|x^k-z^*\|^2\}$ is convergent. 
	Hence, $\{x^k\}$ is bounded. 
	
	Let $\{x^{k_j}\}$ be a subsequence of $\{x^k\}$ such that $x^{k_j}$ converges to some  point $\overline{x}$ and
	\begin{equation}
		\lim_{j \rightarrow +\infty} \langle g^{k_j}, x^{k_j}-z^*\rangle=\liminf_{k \rightarrow +\infty} \langle g^k, x^k -z^*\rangle =0.\label{eq17}
	\end{equation}
	
	Since the sequence $\{y^{k_j}\}_j$ is bounded.   $\{z^{k_j}\}_j$ is  bounded too.  By taking subsequences if necessary, without loss of generality, we can assume that ${y^{k_j}}$ converges to $\overline{y}$ and ${z^{k_j}}$ converges to $\overline{z}$.
	\vskip0.5cm
	\textbf{Step 1:} We will prove that 
	\begin{equation}
		f(\overline{z}, \overline{x})= 0.
	\end{equation}

Indeed, from part (ii) Proposition \ref{prop3}, $f(z^k,x^k)>0$. In addition, by Assumption (A0), $f(.,.)$ is continuous on $C\times C$, we have
$$f(\overline{z}, \overline{x})= \lim_{j \rightarrow +\infty} f(z^{k_j},x^{k_j})\geq 0.$$

Now, assume that  $f(\overline{z}, \overline{x})=\epsilon>0$. Then there exists $j_0$ such that $ f(z^{k_j},x^{k_j})>\frac{\epsilon}{2}$ for all $j\geq j_0$.

Since $z^*\in S_d$, we have $f(\overline{z}, z^*)\leq 0$. Again by (A0), there exists $\epsilon_1, \epsilon_2 >0$ such that for all $z\in B(\overline{z}, \epsilon_1)$, $z'\in B(z^*,\epsilon_2)$:
$$f(z,z') <\frac{\epsilon}{2}.$$
 Since  $\{z^{k_j}\}$ converges to $\overline{z}$,   there exists $j_1$ such that for any $j\geq j_1$  we have $z^{k_j} \in 
B(\overline{z}, \epsilon_1)$, from which it follows that  for $j \geq \max(j_0,j_1)$, and $z'\in B(z^*,\epsilon_2)$, we have
$$f(z^{k_j}, z') < f(z^{k_j},x^{k_j}).$$

By taking $z'=z^*+\epsilon_2 g^{k_j}$ and thanks to (\ref{alg3}), we have for $j \geq \max(j_0,j_1)$,
$$\langle g^{k_j}, x^{k_j} -z^* \rangle >\epsilon_2,$$
which contracts to (\ref{eq17}). Thus $	f(\overline{z}, \overline{x})=0$.

\vskip0.5cm\textbf{Step 2:} We  prove that $$\lim_{j\rightarrow +\infty}\|x^{k_j} -y^{k_j}\|=0.$$
From (\ref{alg2}), we know that for any $k$
$$f(z^k,x^k) >f(z^k,y^k).$$
So, $f(\overline{z}, \overline{y})\leq f(\overline{z}, \overline{x})=0$ (by Step 1). 

Let $\theta_j=\theta^{m_j}$ such that $z^{k_j}= (1-\theta_j)x^{k_j}+\theta_j y^{k_j}$. Clearly, $0<\theta\leq \theta_j<1$. Therefore, $\overline{z}$ is a convex combination of $\overline{x}$ and $\overline{y}$ and $\overline{z}\not= \overline{x}$. 

Now if  $f(\overline{z}, \overline{y})<0$, then $\overline{z}\not= \overline{y}$.  By the semi-strictly quasiconvexity of $f(\overline{z},.)$, we have
$$0=f(\overline{z}, \overline{z})< f(\overline{z}, \overline{x}),$$
which is impossible. It implies that
$$f(\overline{z}, \overline{y})=0.$$

	From (\ref{alg2}),
	\begin{equation}
		f(z^{k_j},x^{k_j}) -f(z^{k_j},y^{k_j}) \geq \alpha\rho_{k_j} \|y^{k_j}-x^{k_j}\|^2.
	\end{equation}
	Let $j \rightarrow +\infty$, and note that $\lim_{k \rightarrow +\infty} \rho_k=\overline{\rho}>0$, we obtain
	$$\lim_{j\rightarrow +\infty}\|x^{k_j} -y^{k_j}\|=0.$$ This means $\overline{x}=\overline{y}$.
	Note that \begin{equation}
		y^{k_j} \in \text{argmin}_{y\in C} \left\{ f(x^{k_j},y) +\frac{1}{2\rho_{k_j}} \|x^{k_j}-y\|^2\right\}.
		\nonumber
	\end{equation}
	then for any $y\in C$, we have
	$$f(x^{k_j},y^{k_j}) +\frac{1}{2\rho_{k_j} }\|x^{k_j}-y^{k_j}\|^2\leq f(x^{k_j},y) +
\frac{1}{2\rho_{k_j}} \|x^{k_j}-y\|^2.$$
	Let $j \rightarrow +\infty$, by the continuity of $f$ and $\overline{x}=\overline{y}$, we obtain for any $y\in C$,
	\begin{equation}\label{M2}
0\leq f(\overline{x},y) +\frac{1}{2\overline{\rho} }\|\overline{x}-y\|^2.
\end{equation}
	Assume, in addition, that $f(x,.)$ is strongly quasiconvex on $C$ with modulus $\gamma >0$.
For any $x\in C$ and $\lambda \in [0,1]$, take $y=\lambda x + (1-\lambda) \overline{x}$. By eplacing it to (\ref{M2}) we obtain
$$0\leq f(\overline{x}, \lambda x+(1-\lambda) \overline{x}) +\frac{1}{2\overline{\rho} }\|\overline{x}-(\lambda x+(1-\lambda) \overline{x}\|^2.$$
Then using the definition of strong quasiconvexity, by the same argument as in the proof of part (ii) in Proposition 3.1,
 we can see that $f(\overline{x} , x) \geq 0$ for every $x\in C$.

\end{proof}
Let 
$$ S_{\overline{\rho}}:=f(\overline{x},y) +\frac{1}{2\overline{\rho} }\|\overline{x}-y\|^2.$$

\begin{remark} (i) In virtue of Lemma \ref{M1} we have  $S_d\subseteq S$. Thus, if $S_{\overline{\rho}}=S_d$, then     $  \overline{x} \in S_d =S$. Remember that  the sequence $\{\|x^k -\overline{x}\|^2\}$ is convergent we can conclude  that the whole sequence $\{x^k\}$ converges to $\overline{x}$ which is a solution of (EP).

(ii)  Since  for any $\overline{\rho}$, one can choose a sequence $\{\rho_k\}$ such that $\rho_k \to \overline{\rho}$. Thus, from  $  f(\overline{x},y) +\overline{\rho} \|\overline{x}-y\|^2 \geq 0 \  \forall y\in C$, it can be seen that  for any $\epsilon >0$, there exists $ \overline{\rho} > 0$ small enough such that
	$  f(\overline{x},y) \geq -\epsilon $ provided  $C$ is bounded. So	 
	 one can considered $ \overline{x}$ as approximate solution.  
 
In the case $f$ is pseudomonotone,  then by Lemma \ref{lem1}   $S=S_d$. Clearly,  $S\subseteq S_{\rho}$ for every $\rho > 0$.   In addition, if $f(x,.)$ is pseudoconvex and continuously differentiable      for any $x\in C$ , then $S = S_{\rho}$ for every $\rho >0$. Hence 
$\overline{x}$ is a solution.

 (iii) Clearly, Assumption (A3) may be dropped if $C$ is bounded (often in practice), moreover, from the proof one can see that this assumption is not needed  if the optimization problem (\ref{alg1}) admits a unique solution for every $k$.

	\end{remark} 
 The following simple example shows that a $\rho$- quasi-solution with any $\rho > 0$ may not be a solution.
 	
Let $C  := [-1, 0]$,  $f(x,y) := y^3 -x^3$, Clearly, with $x^* =0$, we have $f(x^*,y) + \frac{1}{2\rho}(y-x^*)^2 = y^3  +\frac{1}{2\rho}y^2 \geq 0, \ \forall y\in C$  if $\rho>0$ small enough, for example $\rho< 1/2$. Thus, 
$0$ is $\rho$- prox-solution, but $f(x^*,y)  = y^3 < 0$ with $y=-1\in C$. So the auxiliary problem principle fails to appy to semi-strictly quasiconvex equilibrium problems,

 Now we consider an  example \cite{KT1998} in which the  optimization problem (\ref{alg1})  can be solved efficiently. 
 
 Suppose that  bifunction $f(x,y):= \max_{i\in I}g_i(x,y)$, where $I \subset \R$ is compact, and each $g_i(x,.)$ is quasiconvex on $C$ for every fixed $x\in C$. Then $f(x,.)$ is quasiconvex.
Suppose that   each $g_i(x,.)$ ($i\in I)$  is differentiable and  its derivative  is Lipschitz with constant $L_i(x)>0.$
  Let $\rho <\frac{1}{L(x)} $ with $L(x):= \max_{i\in I}L_i(x)$ and 
  $$f_\rho(x,y) := f(x,y) +\frac{1}{2\rho}\|y-x\|^2.$$
  Then $f_\rho(x,.)$ is strongly convex on $C$. Indeed, for any $u,v\in C$ and $i\in I$, consider the function 
 $g_{i,\rho}(x,y):= g_i(x,y) +\frac{1}{2\rho}\|y-x\|^2$. Then we have
  $$\langle \nabla_2 g_{i, \rho}(x,u) -\nabla_2 g_{i,\rho}(x,v),u-v\rangle$$
  $$= \langle \nabla_2 g_i(x,u) +\frac{1}{\rho}  (u-x) -\nabla_2 g_i(x,v) - \frac{1}{\rho} (v-x),u-v\rangle $$
  $$\geq -L(x)\|u-v\|^2 +\frac{1}{\rho}\|u-v\|^2 = (\frac{1}{\rho}- L(x))\|u-v\|^2.$$
  Hence $f_\rho(x,.)$ is strongly convex whenever $\rho <\frac{1}{L(x)}$. The above example belongs to the class of the lower -$C^2$ functions considered  by some some authors see e.g.
\cite{Mi1977,RW1998,Vi1983}.

\section{Numerical experiments}
We present here two examples two illustrate the behavior of our linesearch extragradient algorithm (for short LEQEP) for quasiconvex equilibrium problems. The algorithm is implemeted in Python 3 running on a Laptop with AMD Ryzen 7 5800H with Radeon Graphics 3.20 GHz and 8GB RAM memory.
\begin{example}\label{ex1}
We consider the following $1$-dim strongly quasiconvex equilibrium problems (\cite[Example 4.2]{IL2021})
$$C=\left[0, \delta\right];$$
$$f(x,y)=\sqrt{y}-\sqrt{x} +r x(y-x);$$
where $\alpha, \delta >0$. It is easy to see that the solution set is $S=\{0\}.$

We test LEQEP on this example with $r= 2$ and $\delta=10$. We take $x^0=5$, $\rho=1$, $\alpha=0.5$, $\theta=0.5$ and stop the algorithm if $|x^k -y^k| <10^{-3}$ or $|x^{k+1}-x^{k}|<10^{-3}$. Our algorithm  reach the unique solution $x^*=0$ after $84$ iterations if we choose $\sigma_k=\frac{1}{k+1}$ and $8$ iterations if we choose $\sigma_k=\frac{2}{k+1}$. Figure \ref{fig1} illustrates the behavior of LEQEP for this example.

\begin{figure}[h!]
     \centering
     \begin{subfigure}[b]{0.45\textwidth}
         \centering
         \includegraphics[width=\textwidth]{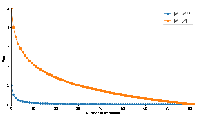}
         \caption{$\sigma_k=\frac{1}{k+1}$}
     \end{subfigure}
     \hfill
     \begin{subfigure}[b]{0.45\textwidth}
         \centering
         \includegraphics[width=\textwidth]{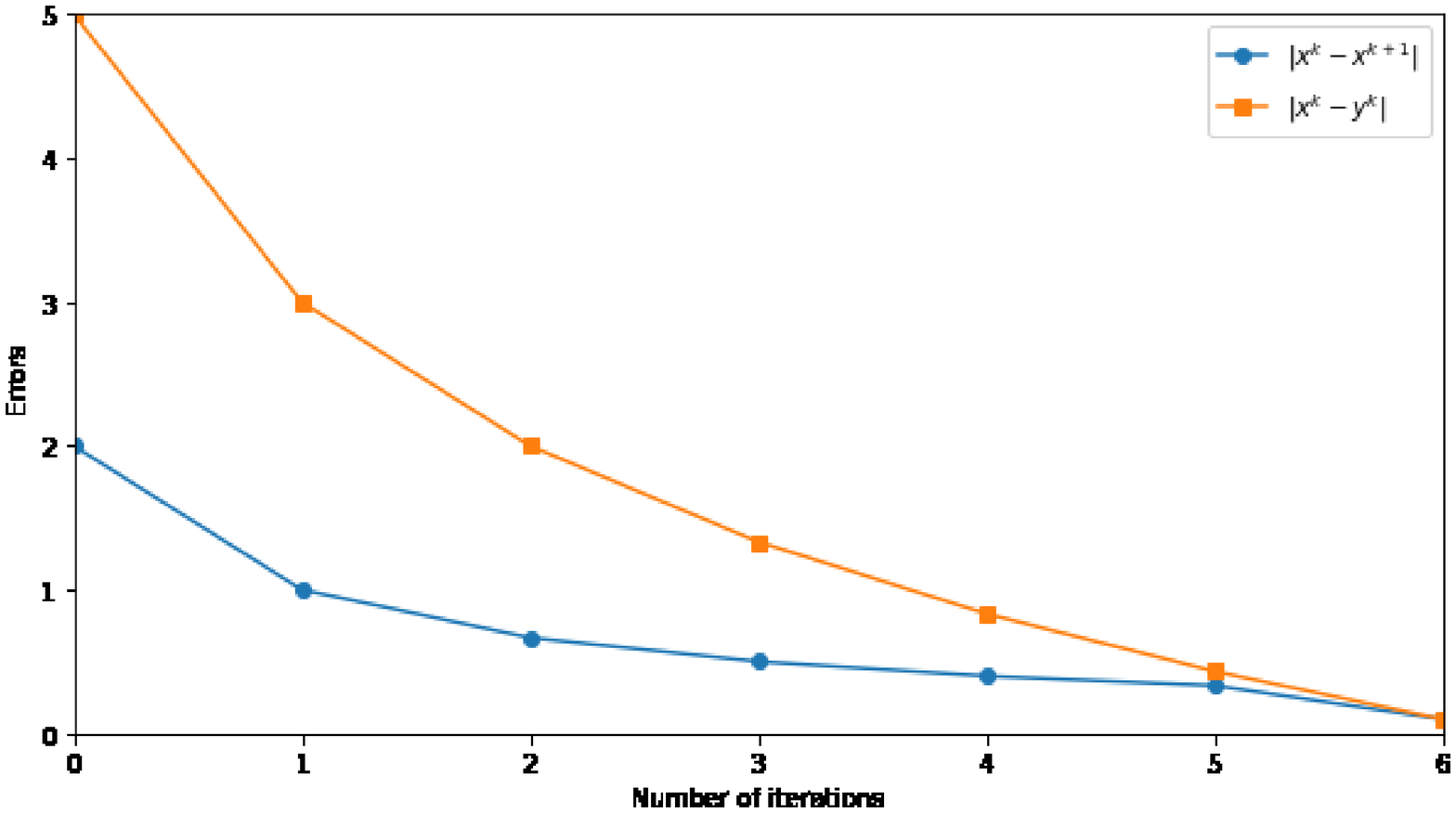}
         \caption{$\sigma_k=\frac{2}{k+1}$}
     \end{subfigure}
     
        \caption{Behavior of LEQEP in Example \ref{ex1}}
        \label{fig1}
\end{figure}
\end{example}

\begin{example} \label{ex2}
We consider the bifunction
$$f(x,y)=\max\{f_1(x,y),f_2(x,y)\},$$
where
$$f_1(x,y)=\langle A_1x+b_1, \frac{E_1y+f_1}{c_1^Ty+d_1}- \frac{E_1x+f_1}{c_1^Tx+d_1}\rangle,$$
$$f_2(x,y)=\langle A_2x+b_2, \frac{E_2y+f_2}{c_2^Ty+d_2}- \frac{E_2x+f_2}{c_2^Tx+d_2}\rangle,$$
with $A_1,A_2,E_1,E_2 \in \mathbb{R}^{m\times n}$, $b_1,b_2,f_1,f_2\in \mathbb{R}^m$, $c_1,c_2\in \mathbb{R}^n$ and $d_1,d_2 \in \mathbb{R}$. We also assume that 
$$C{\color{blue}\subset}\{x: c_1^Tx+d_1>0\}\cap \{x: c_2^Tx+d_2>0\}.$$

In the first experiment, we take $m=n=2$ and
$$A_1=A_2=I_2, \quad b_1=b_2 =0_2,$$
$$E_1=I_2, \quad f_1=1_2,\quad c_1=0_2,\quad d_1=1,$$
$$E_2=\begin{bmatrix}
1&2\\3&4
\end{bmatrix}, \quad f_2=c_2=1_2, \quad d_2=2,$$
$$C=\left[0,5\right]^2.$$

We test LEQEP on this example with  $x^0=\begin{bmatrix}
5&5
\end{bmatrix}^T$, $\rho=1$, $\alpha=0.5$, $\theta=0.8$ and stop the algorithm if $|x^k -y^k| <10^{-5}$ or $|x^{k+1}-x^{k}|<10^{-5}$. Our algorithm  reach an approximate solution of $x^*=\begin{bmatrix}0& 0 \end{bmatrix}^T$ after $129$ iterations if we take $\sigma_k=\frac{1}{k+1}$ and after $13$ iterations if we take $\sigma_k=\frac{3}{2(k+1)}$. Figure \ref{fig2} illustrates the behavior of LEQEP in this example.
\begin{figure}[h!]
     \centering
     \begin{subfigure}[b]{0.45\textwidth}
         \centering
         \includegraphics[width=\textwidth]{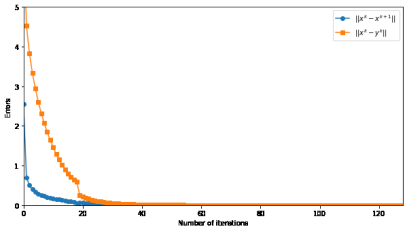}
         \caption{$\sigma_k=\frac{1}{k+1}$}
     \end{subfigure}
     \hfill
     \begin{subfigure}[b]{0.45\textwidth}
         \centering
         \includegraphics[width=\textwidth]{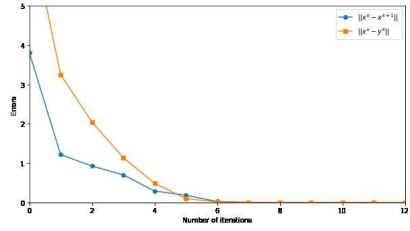}
         \caption{$\sigma_k=\frac{3}{2(k+1)}$}
     \end{subfigure}
\caption{Behavior of LEQEP in Example \ref{ex2} for $m=n=2$}
\label{fig2}
\end{figure}
In the second experiment, we take $m=n=10$ and
$$A_1=A_2=I_{10}, \quad b_1=b_2 =0_{10},$$
$$E_1=I_{10}, \quad f_1=1_{10},\quad c_1=0_{10},\quad d_1=1,$$
$$E_2=\begin{bmatrix}
4 &4 &2& 4& 2& 2& 3& 2& 1& 2\\
1& 3& 0& 3& 0& 1& 1& 0& 2& 2\\
2& 3& 0 & 2& 2& 2& 1& 1& 2& 2\\
3& 2& 3& 0& 1& 1& 2& 4& 1& 1\\
1& 4& 3& 0& 2& 1& 4& 3& 0& 3\\
0& 4& 1& 4& 2& 3& 4& 3& 4& 2\\
3& 0& 4& 4& 0& 4& 1& 1& 1& 2\\
1& 2& 2& 3& 1& 0& 3& 0& 0& 0\\
2& 0& 0& 3& 0& 3& 3& 4& 4& 2\\
0& 2& 4& 4& 0& 4& 3& 0& 3& 1
\end{bmatrix}, \quad f_2=c_2=1_{10}, \quad d_2=2,$$
$$C=\left[0,5\right]^2.$$

We test LEQEP on this example with  $x^0=5*1_{10}$, $\rho=1$, $\alpha=0.5$, $\theta=0.8$ and stop the algorithm if $|x^k -y^k| <10^{-5}$ or $|x^{k+1}-x^{k}|<10^{-5}$. Figure \ref{fig3} illustrates the behavior of LEQEP in this example.
\begin{figure}[h!]
     \centering
     \begin{subfigure}[b]{0.45\textwidth}
         \centering
         \includegraphics[width=\textwidth]{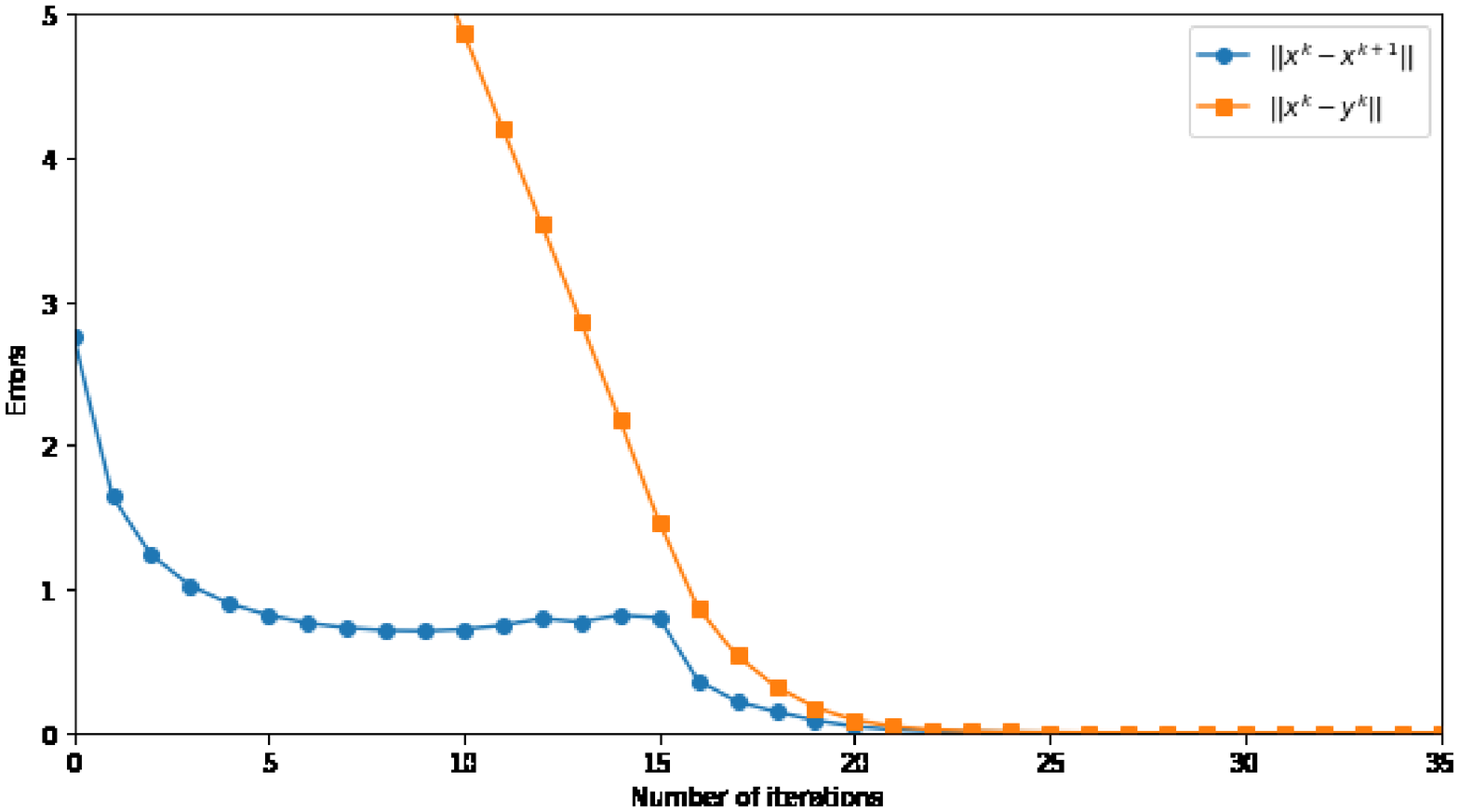}
         \caption{$\sigma_k=\frac{1}{k+1}$}
     \end{subfigure}
     \hfill
     \begin{subfigure}[b]{0.45\textwidth}
         \centering
         \includegraphics[width=\textwidth]{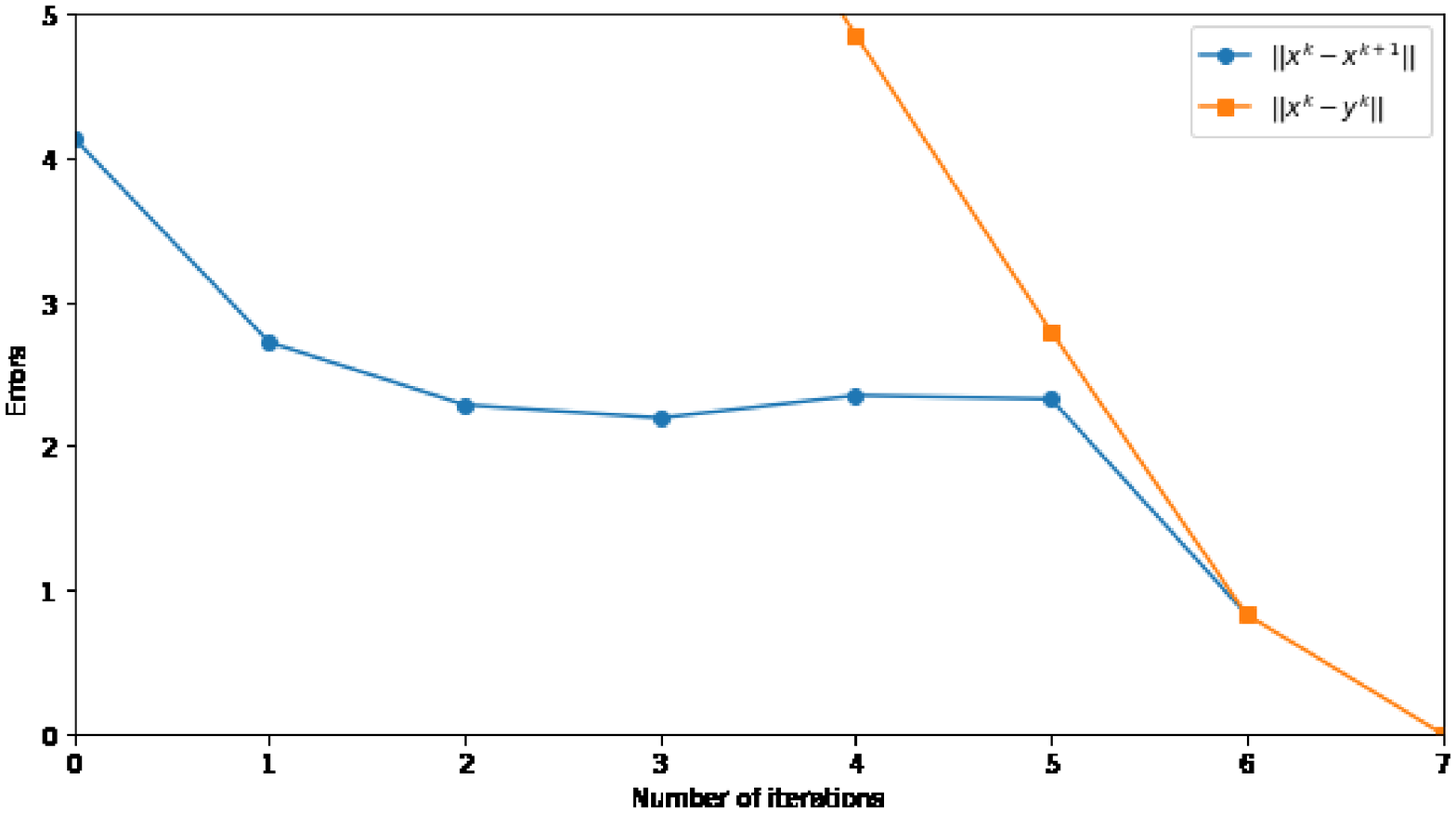}
         \caption{$\sigma_k=\frac{3}{2(k+1)}$}
     \end{subfigure}
\caption{Behavior of LEQEP in Example \ref{ex2} for $m=n=10$}
\label{fig3}
\end{figure}

In the last experiment, each entry of the matrices $A_1,A_2,E_1,E_2$, vectors $b_1,b_2$, $c_1,c_2$, $f_1,f_2$ and number $d_1,d_2$ is randomly generated in the interval $\left[0,5\right]$. We test LEQEP for $m=n= 5,10,20,50$,  $x^0=5*1_{n}$, $\rho=1$, $\alpha=0.5$, $\theta=0.8$  and stop the algorithm if $|x^k -y^k| <10^{-8}$ or $|x^{k+1}-x^{k}|<10^{-8}$ or the number of iterations exceed $1000$. The average time and average error $\min (\|x^k-y^k\|, \|x^k -x^{k+1}\|)$ for each size are reported in Tables \ref{tab1} with different sizes, a hundred of problems have been tested for each size.
\begin{table}[ht]
\caption{Algorithm with $\alpha_k=\frac{n}{k+1}$} 
\centering 
\begin{tabular}{ c c l l } 
\hline\noalign{\smallskip}
 n &N. of prob.   & CPU-times(s)& Error\\
\noalign{\smallskip}\hline\noalign{\smallskip}
5&100&0.0018502084016799928&3.0083411036103923e-05 \\
10&100&0.002031816816329956&6.697999417359917e-06  \\
20&100&0.0020687472820281982&4.7459739660895515e-06 \\
50&100&0.011073581838607788&0.004554847325973183 \\
\hline\noalign{\smallskip}
\end{tabular}
\label{tab1} 
\end{table}
We also record the average errors of $\\|x^k-y^k\|$ and $\|x^k -x^{k+1}\|$ in the first $1000$ iterations in Figure \ref{fig4}. 
\begin{figure}
 	\centering
     \begin{subfigure}[b]{0.45\textwidth}
         \centering
         \includegraphics[width=\textwidth]{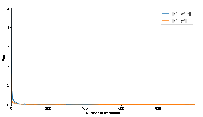}
         \caption{$m=n=5$}
     \end{subfigure}
     \hfill
     \centering
     \begin{subfigure}[b]{0.45\textwidth}
         \centering
         \includegraphics[width=\textwidth]{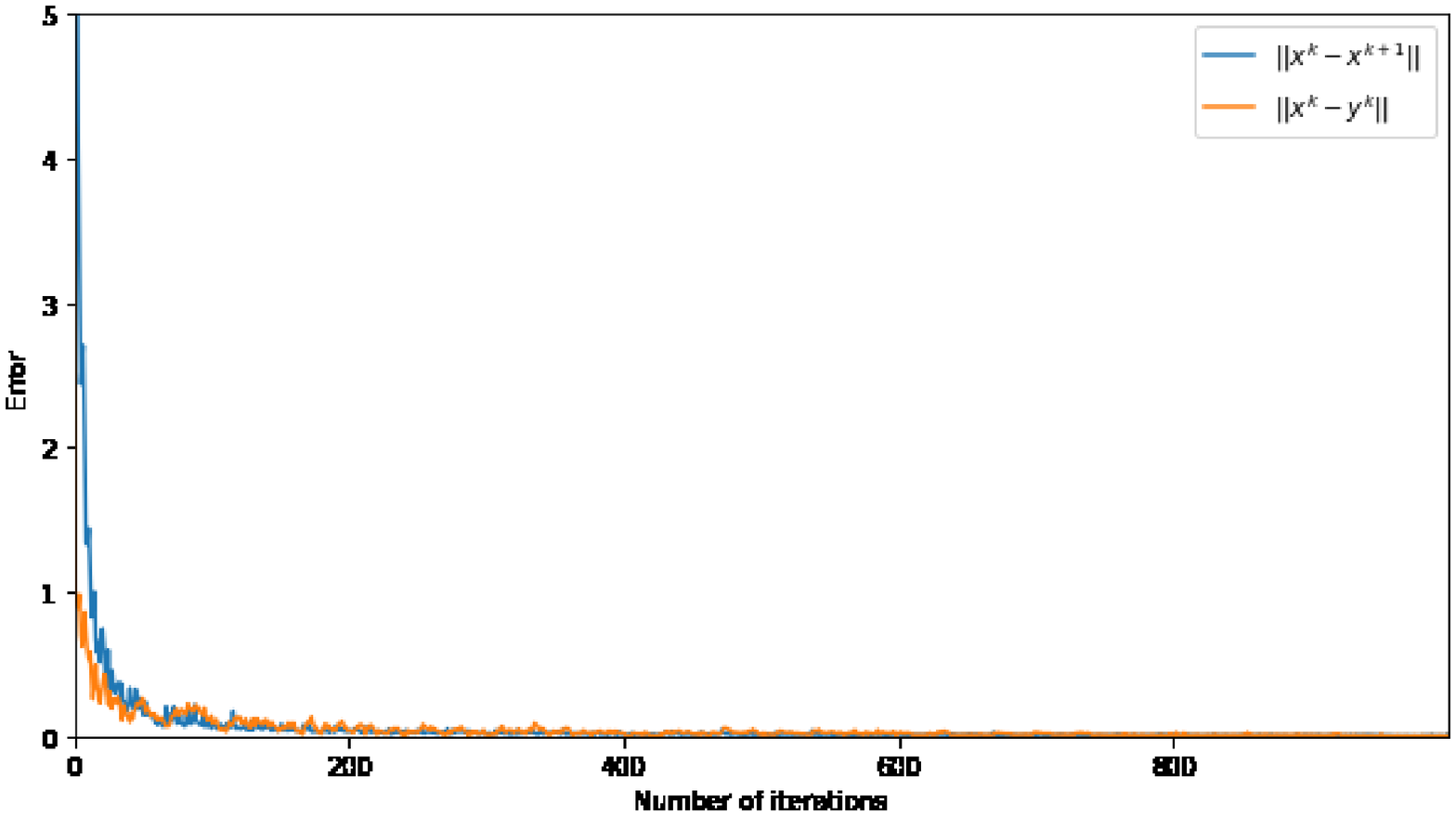}
         \caption{$m=n=10$}
     \end{subfigure}
     \hfill
     \begin{subfigure}[b]{0.45\textwidth}
         \centering
         \includegraphics[width=\textwidth]{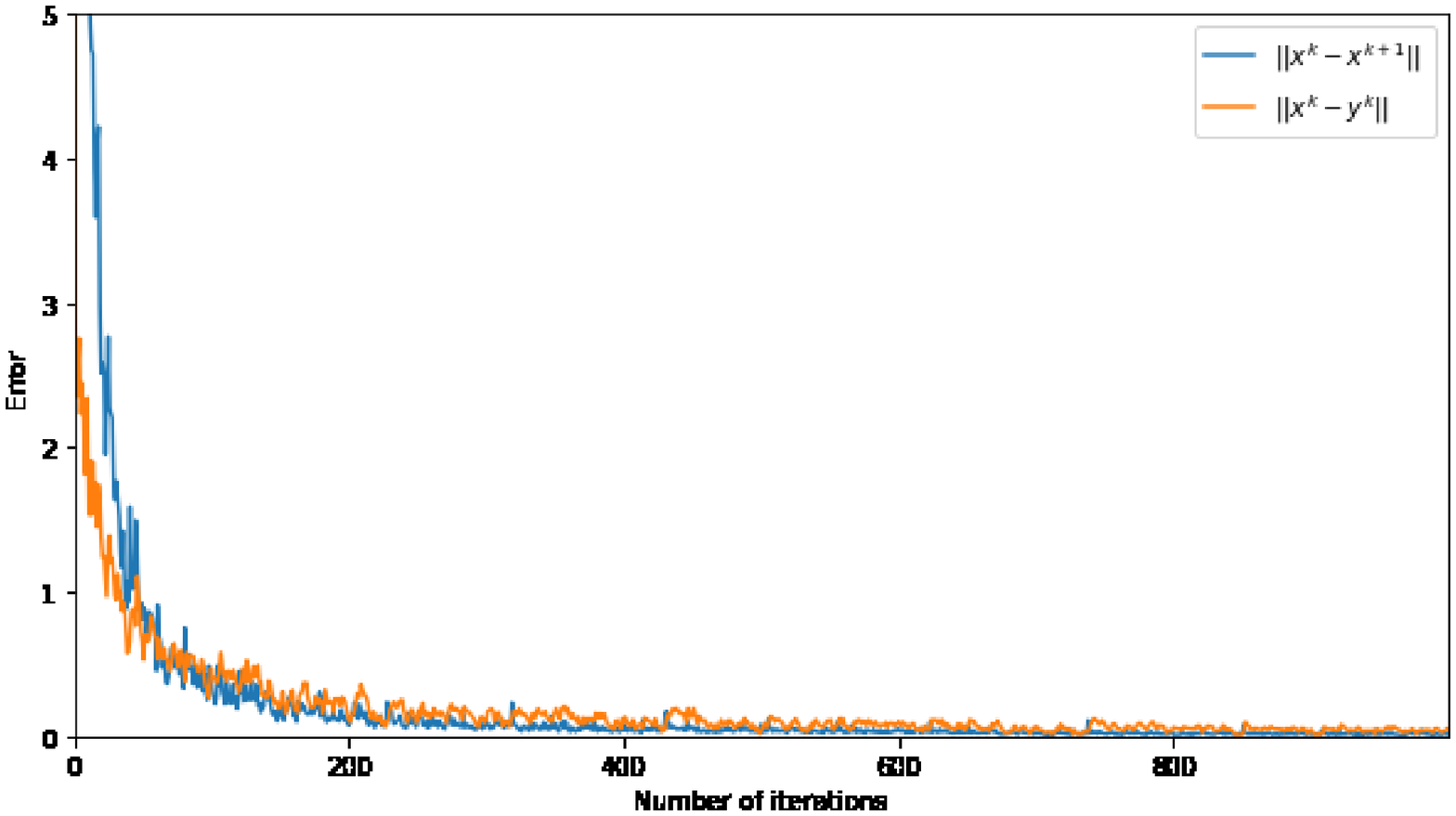}
         \caption{$m=n=20$}
     \end{subfigure}
     \hfill
     \begin{subfigure}[b]{0.45\textwidth}
         \centering
         \includegraphics[width=\textwidth]{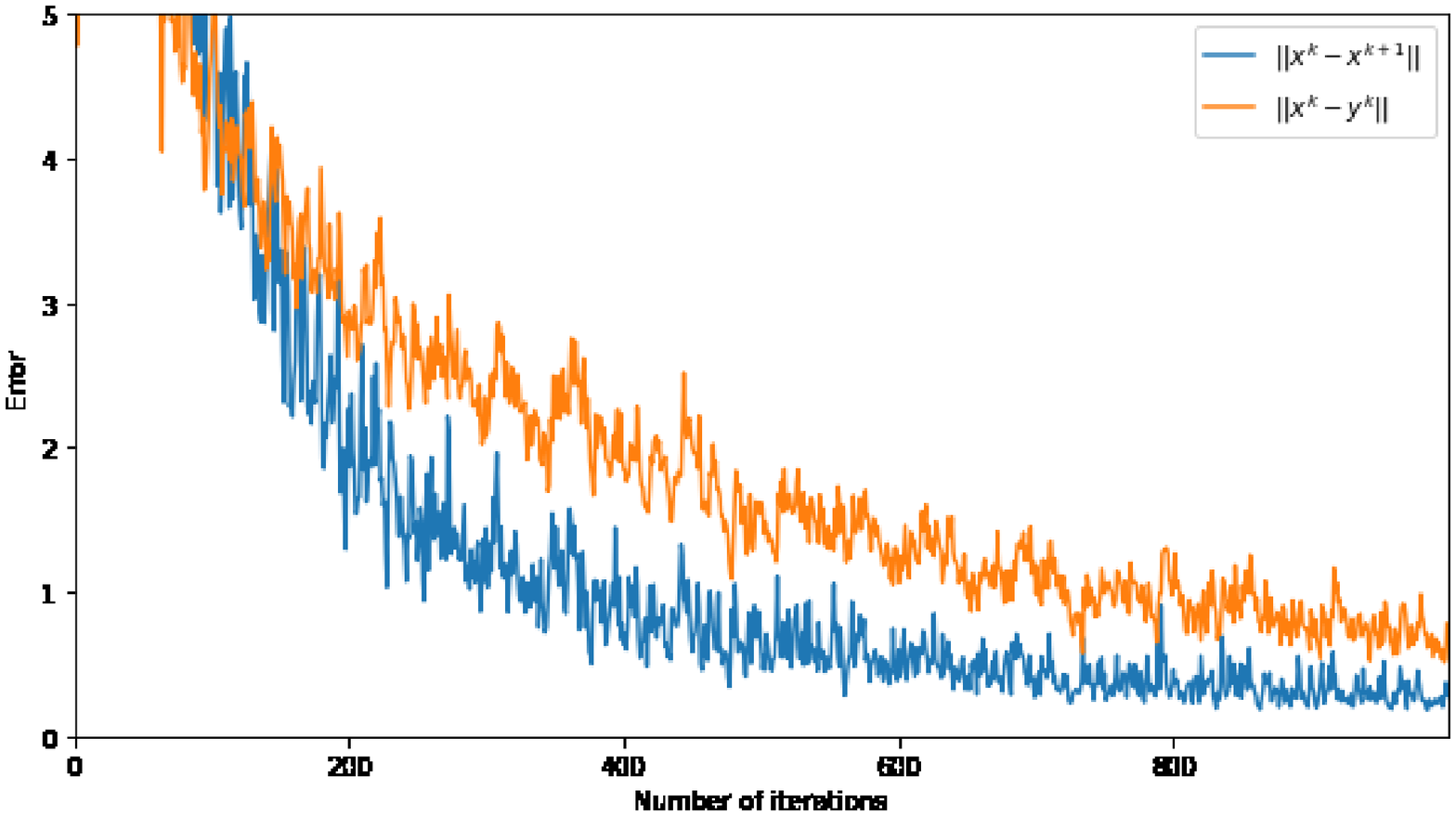}
         \caption{$m=n=50$}
     \end{subfigure}
     \caption{Behavior of LEQEP for random input}
\label{fig4}
\end{figure}

\end{example}

\section*{Conclusion.} We have  proposed an extragradient linesearch algorithm for approximating a solution of  equilibrium problems with quasiconvex bifunctions. The sequence of the iterates generated by the proposed algorithm converges to a proximal-solution when the bifunction is semi-strictly quasiconvex with respect to its second variable, which is an equilibrium solution provided  the bifunction is strongly quasiconvex. Neither monotonicity nor Lischitz properties are required. Thus the algorithm could be considered as an iterative scheme for a solution of the problem considered by K. Fan in \cite{Fa1972} with the bifunction being semi-strictly quasiconvex.


%
%



\end{document}